\documentclass[10pt]{amsart}
\usepackage{amsfonts}
\usepackage{epsfig}
\usepackage{graphics}
\usepackage{graphicx}
\def\vbar{\mathchoice{\vrule height6.3ptdepth-.5ptwidth.8pt\kern-.8pt}
   {\vrule height6.3ptdepth-.5ptwidth.8pt\kern-.8pt}
   {\vrule height4.1ptdepth-.35ptwidth.6pt\kern-.6pt}
   {\vrule height3.1ptdepth-.25ptwidth.5pt\kern-.5pt}}
\def\fudge{\mathchoice{}{}{\mkern.5mu}{\mkern.8mu}}
\def\bbc#1#2{{\rm \mkern#2mu\vbar\mkern-#2mu#1}}
\def\bbb#1{{\rm I\mkern-3.5mu #1}}
\def\bba#1#2{{\rm #1\mkern-#2mu\fudge #1}}
\def\bb#1{{\count4=`#1 \advance\count4by-64 \ifcase\count4\or\bba
A{11.5}\or \bbb B\or\bbc C{5}\or\bbb D\or\bbb E\or\bbb F \or\bbc
G{5}\or\bbb H\or \bbb I\or\bbc J{3}\or\bbb K\or\bbb L \or\bbb
M\or\bbb N\or\bbc O{5} \or \bbb P\or\bbc Q{5}\or\bbb R\or\bbc
S{4.2}\or\bba T{10.5}\or\bbc U{5}\or \bba V{12}\or\bba
W{16.5}\or\bba X{11}\or\bba Y{11.7}\or\bba Z{7.5}\fi}}
\newtheorem{theo}{Theorem}[section]
\newtheorem{lemma}{Lemma}[section]
\newtheorem{remark}{Remark}[section]
\newtheorem{prop}{Proposition}[section]

\def \d {\displaystyle}

\def \barr {\begin{array}{l}}
\def \ear {\end{array}}
\def \beq {\begin{equation}}
\def \eeq {\end{equation}}
\def \beqn {\begin{eqnarray}}
\def \eeqn {\end{eqnarray}}
\def \beqn* {\begin{eqnarray*}}
\def \eeqn* {\end{eqnarray*}}
\def \f {\end{document}}

\newcommand{\field}[1]{\mathbb{#1}}
\newcommand{\R}{\field{R}}

\newcommand{\N}{\field{N}}

\usepackage{texdraw,color}

\begin{document}

\thispagestyle{empty}

\title[Nonlinear feedback stabilization of the disk-beam system]
{Improved results on the nonlinear feedback stabilization of a rotating body-beam system}

\author{Ka\"{\i}s Ammari}
\address{UR Analysis and Control of PDEs, UR13ES64, Department of Mathematics, Faculty of Sciences of Monastir, University of Monastir, 5019 Monastir, Tunisia}
\email{kais.ammari@fsm.rnu.tn}

\author{Ahmed Bchatnia}
\address{UR  Analyse Non-Lin\'eaire et G\'eom\'etrie, UR13ES32, Department of Mathematics, Faculty of Sciences of Tunis, University of Tunis El Manar, Tunisia}
\email{ahmed.bchatnia@fst.utm.tn}

\author{Boumedi\`ene Chentouf$^*$}
\address{Kuwait University, Faculty of Science, Department of Mathematics, Safat 13060, Kuwait}
\email{boumediene.chentouf@ku.edu.kw; chenboum@hotmail.com($^*$corresponding author)}

\begin{abstract} This article is dedicated to the investigation of the stabilization problem of a flexible beam attached to the center of a rotating disk. We assume that the feedback law contains a nonlinear torque control applied on the disk and nonlinear boundary controls exerted on the beam. Thereafter, it is proved that the proposed controls guarantee the exponential stability of the system under a realistic smallness condition on the angular velocity of the disk and general assumptions on the nonlinear functions governing the controls. We used here the strategy of Lasiecka and Tataru \cite{lastat} and Alabau-Boussouira \cite{AB1, AB2}. This permits to improve the stability result shown in \cite{CH:99} in the sense that, on one hand, we deal with a general form of the nonlinear functions involved in the boundary controls. On the other hand, we  manage to weaken the conditions on those functions unlike in \cite{CH:99}, where the authors consider a special type of functions that are almost linear.

\end{abstract}

\maketitle



\section{Introduction}
The stabilization problem of coupled elastic and rigid parts systems has attracted a huge amount of interest and in particular the rotating disk-beam system which arises in the study of large-scale flexible space structures \cite{Ba:87}. It consists of a flexible beam (B), which models a flexible robot arm, clamped at one end to the center of a disk (D) and free at the other end (see Figure  \ref{fig0}). Additionally, it is assumed that the center of mass of the disk is fixed in an inertial frame and rotates in that frame with a non-uniform angular velocity. Under the above assumptions, the motion of the whole structure is governed by the following nonlinear hybrid system (see \cite{Ba:87} for more details):

\begin{equation}\label{1n}
\left \lbrace
\begin{array}{ll}
\rho y_{tt}(x,t) + EI y_{xxxx}(x,t) = \rho \omega^2 (t) \, y(x,t), & (x,t) \in (0,\ell) \times (0,\infty),\\
y(0,t) = y_x (0,t) =0, & t>0, \\
y_{xxx} (\ell,t)  = \d {\mathcal F} (t), & t>0,  \\
y_{xx} (\ell,t)={\mathcal M} (t), & t>0, \\
\displaystyle \frac{\textstyle d }{dt}  \left\{ \omega (t) \left( {\textstyle I_d + \displaystyle \int_0^{\ell} \rho y^2 (x,t) \, dx} \right) \right \}=  {\mathcal T} (t) , & t>0,\\
y(x,0)=y_0(x), \; y_t(x,0)=y_1(x), & x \in (0,\ell),\\
\omega(0)=\omega_0 \in \R ,
\end{array}
\right.
\end{equation}
in which $y$ is the beam's displacement in the rotating plane  at time $t$ with respect to the spatial variable $x$ and $\omega $ is the angular velocity of the disk. Furthermore, $\ell$ is the length of the beam and the physical parameters $EI, \rho$ and $I_d$ are respectively the flexural rigidity, the mass per unit length of the beam, and the disk's moment of inertia. Lastly, ${\mathcal M} (t)$ and ${\mathcal F} (t)$ are  respectively the moment and force control exerted at the free-end of the beam, while the torque control ${\mathcal T} (t)$ acts on the disk.

\begin{figure}[ht]
\includegraphics[width=0.45 \textwidth]{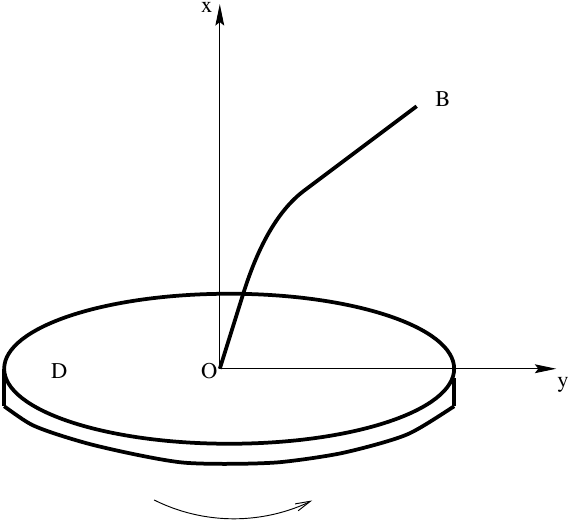}
\caption{The disk-beam system}
\label{fig0}
\end{figure}


After the pioneer work of Baillieul \&  Levi  \cite{Ba:87}, the stabilization problem of the rotating system has been the object of a considerable mathematical endeavor (see for instance \cite{bt:90}, \cite{bc1}-\cite{bc8}, \cite{ac:98, gu, lxs:96, mor1, Mo1:91, mor2, Mo2:94, xb:93}, and the references therein).

Instead of surveying the vast literature on this subject, we are going to highlight only those closely related to the context of our main concern, which is {\bf nonlinear} stabilization of the system. In fact, to the best of our knowledge, very few works have been devoted to nonlinear stabilization of the system (\ref{1n}). Specifically, the authors in \cite{ac:98} constructed a nonlinear feedback torque control law ${\mathcal T} (t)$, which globally asymptotically stabilizes the system. Later, an exponential stabilization result has been established in \cite{CH:99} via a class of nonlinear boundary and torque controls defined below
\begin{equation}
\left \{
\barr
{\mathcal F}  (t)= g \left( y_t(1,t) \right),\\
{\mathcal M} (t)=-f(y_{tx}(1,t)),\\
{\mathcal T}  (t)=-\gamma(\omega (t)-\varpi),
\ear
\right.
\label{(F)}
\end{equation}
where $\varpi \in \R $ and $f, g$ and $\gamma$ are nonlinear real functions. In fact, the results in \cite{CH:99} is obtained under very restrictive conditions, which we can summarize as follows: $f$ and $g$ are almost everywhere linear functions (see \cite{CH:99} for more details). We also note that in \cite{guo}, the authors proposed a linear torque control and a {\bf nonlinear interior control of type viscous damping} to obtain the exponential stability of the closed loop system.

\smallskip

It is worth mentioning that the boundary stabilization of elastic systems has been the subject of extensive works. For instance, the case of the wave equation with nonlinear boundary feedback has been considered in \cite{Las, wa} for the asymptotic  stability properties (see also \cite{zu}). We also note that the question of the uniform stabilization or pointwise stabilization of the Euler-Bernoulli beams or serially connected beams has been studied by a number of authors (cf. F. Alabau-Boussouira \cite{AB1, AB2, AB3}, K. Ammari and M. Tucsnak \cite{amm1}, K. Ammari and Z. Liu \cite{fi1}, K. Ammari \cite{fi2}, K. Ammari et al. \cite{fi3}, K. Ammari and S. Nicaise \cite{amm2}, M. A. Ayadi, A. Bchatnia \cite{AY}, M. A. Ayadi, A. Bchatnia, M. Hamouda and S. Messaoudi \cite{ABHM}, G. Chen, M. C. Delfour, A. M. Krall and G. Payre \cite{chen}, J. L. Lions \cite{lion}, K. Nagaya \cite{nag}, Lasiecka and Tataru \cite{lastat} etc.).

\smallskip
The novelty of the present work is to provide a more general approach to the stabilization problem (\ref{1n})-(\ref{(F)}) under less restrictive conditions on the nonlinear functions $f$ and $g$. Before going into details, we would like to point out that the presence of the moment control ${\mathcal M} (t)=-f(y_{tx}(1,t))$ in the feedback law (\ref{(F)}) is not required to get the stability result of the closed-loop system but provides more dissipation to the energy of the system (see Lemma 3.1). In other words, one can take ${\mathcal M} (t)=0$ and our stability results will remain valid. This will be made more clear later (see Remark \ref{nrem}). In turn, for sake of simplicity and without loss of generality, we shall take $g=-f$ and show that the closed-loop system possesses the very desirable property of exponential stability as long as the angular velocity is bounded (see (\ref{ang})) and the functions $f$ obeys weaker conditions than those in \cite{CH:99}. Indeed, the stability outcome of \cite{CH:99} becomes a particular case of our findings.

\smallskip
Now, let us briefly present an overview of this paper. In Section \ref{pre}, we formulate the closed-loop system in an abstract first-order evolution equation in an appropriate state space. Section \ref{sta1} deals with the long time behavior (in a general context) of a subsystem associated to the closed-loop system. In section \ref{exa}, some examples are provided. Section \ref{sta} is devoted to the proof of the stability of the closed-loop system. Finally, the paper closes with conclusions and discussions.

\section{Preliminaries}\label{pre}
First, one can use an appropriate change of variables so that $\ell=1$. Then, the closed loop system (\ref{1n})-(\ref{(F)}) can be written as follows:
\begin{equation}
\left \lbrace
\begin{array}{ll}
 \rho y_{tt} + EI y_{xxxx} = \rho \omega ^2 (t) y, & (x,t) \in (0,1) \times (0,\infty), \\
y(0,t) = y_x (0,t) =0, & t >0,\\
y_{xx} (1,t) = -f(y_{tx}(1,t)), & t >0, \\
 y_{xxx} (1,t) = f(y_t(1,t), & t >0, \\
\displaystyle \dot\omega (t) =\displaystyle \frac{-\gamma(\omega(t)-\varpi) -2 \rho
\omega (t) \displaystyle \int_0^1 y(t) y_t(t) dx}{\textstyle I_d +
\rho \displaystyle\int_0^1 y(t)^2 dx }, & t >0,\\
y(x,0)=y_0(x), \; y_t(x,0)=y_1(x), & x \in (0,1),\\
\omega(0)=\omega_0 \in \R . \label{cls}
\end{array}
\right.
\end{equation}
Thereafter, let $z(\cdot,t)=y_t(\cdot,t)$ and consider the state variable
$$ (\phi (t),\omega(t))=(y(\cdot,t), z(\cdot,t),\omega(t)).$$
Next, for $ n \in \N, n \geq 2 $, we denote by
$$ H_l^n=\{ u \in H^n(0,1);\; u(0)=u_x(0)=0 \}.$$
Now, consider the state space ${\mathcal X}$ defined by
$$ {\mathcal X} = \underbrace{H_l^2 \times L^2(0,1)}_{\mathcal H} \times  \R ={\mathcal H}  \times \R, $$
equipped with the following real inner product (the complex case is similar):
\begin{eqnarray}
\langle (y,z,\omega), (\tilde y,\tilde z, \tilde \omega) \rangle_{\mathcal X} &=& \langle (y,z), (\tilde y,\tilde z  )\rangle_{\mathcal H} + \omega\tilde \omega \nonumber \\
&=&  \displaystyle \int_0^{1} \left( EI y_{xx} \tilde y_{xx}- \rho \varpi^2 y \tilde y+ \rho  z \tilde z \right)  dx   +   \omega \tilde \omega.
\label{ip}
\end{eqnarray}
Note that the norm induced by this scalar product is equivalent to the usual one of the Hilbert
space $  H^2 (0,1)\times L^2(0,1) \times \R $ provided that (see \cite{CH:99})
\begin{equation}
\mid  \varpi
\mid < 3\sqrt{EI/ \rho}.
\label{ang}
\end{equation}

\medskip

Now, we define a nonlinear operator $A$ by
\begin{equation}
{\mathcal D}(A) = \left \{ (y,z ) \in H_l^4 \times H_l^2; \;y_{xxx}(1)= f(z(1)), \; y_{xx} (1)= - f( z_x (1))
\right \},
\label{(5)}
\end{equation}

\medskip
and
\begin{equation}
A(y,z) = \left( -z, \frac{EI}{\rho} y_{xxxx}-{ \varpi} ^2 y \right).
\label{(6)}
\end{equation}

Whereupon the system (\ref{cls}) can be formulated in ${\mathcal X}$ as follows
\begin{equation}
\left( \dot \phi(t) , \dot \omega (t) \right)
+ \left( {\mathcal A} +{\mathcal B} \right) \left( \phi(t) , \omega (t) \right) =0,
\label{(6.0)}
\end{equation}
in which
\begin{equation}
\left \lbrace
\begin{array}{l}
{\mathcal A} ( \phi, \omega )=\left(  A \phi, 0  \right), \; \mbox{with} \; {\mathcal D}(\mathcal A)
 = {\mathcal D}(A) \times \R,\\[5mm]
{\mathcal B} ( \phi, \omega ) = \left( 0, ({\varpi} ^2 - \omega ^2 )y, \displaystyle
 \frac{\textstyle \gamma \left( \omega - \varpi \right) + 2 \rho \; \omega (t)
 < y, z >_{L^2 (0,1)}}{\textstyle I_d + \rho \| y\|_{L^2 (0,1)} ^2 } \right).
\end{array}
\right.
\label{(6.1)}
\end{equation}

\medskip

We note that we will adapt here the approach of Lasiecka \& Tataru in \cite{lastat} and of Alabau-Boussouira in \cite{AB1, AB2}. {To proceed,} we suppose that the functions $f$ and $\gamma$ satisfy the following conditions:
\begin{quote}
{\bf H.I:} $f$: $\mathbb{R}\rightarrow \mathbb{R}$ is a non-decreasing differentiable function such that $f(0)=0$.\\
{\bf H.II:} There exists a strictly increasing function $f_{0} \in C([0,+\infty))$ and continuously differentiable in a neighborhood of $0$ such that $f_0(0)=f_0'(0)=0$ and
$$ \left\{
\begin{array}{l}
f_{0}(\vert s \vert)\leq \vert f(s) \vert \leq f_{0}^{-1}(\vert s \vert),%
\hspace{1.65cm} \text{for all}\hspace{1.2em} \vert s\vert <\sigma, \\
c_{1} \vert s \vert\leq \vert f(s)\vert \leq c_{2} \vert s\vert ,\hspace{%
24mm}\text{ for all} \hspace{4mm}\vert s\vert\geq \sigma,
\end{array}
\right.
$$
where $c_{i} > 0$ for i = 1, 2 and for some $\sigma >0.$ \\
{\bf H.III:} The function $\gamma$ is Lipschitz on each bounded subset
of $\R$ and for some $\kappa>0$, 
$$
\gamma(s)s \geq 0,~\;  \mid
\gamma(s)\mid \; \geq \kappa \mid s\mid, \; \forall \,
 s \in \R.
$$
\end{quote}

Moreover, we define a function $H$ by
\begin{equation}\label{210}
 H(x)=\sqrt{x}f_{0}(\sqrt{x}), \, \forall \, x \geq 0.
\end{equation}
Thanks to the assumptions {\bf H.I}-{\bf H.II}, the function $H$ is of class $C^{1}$ and is strictly convex on $(0,r^{2}]$, where $r > 0$ is a sufficiently small number.

\section{Exponential stability of a subsystem}\label{sta1}
In this section, we shall deal with the following subsystem
\begin{equation}
\left \lbrace
\begin{array}{ll}
\rho y_{tt} + EI y_{xxxx} -\rho { \varpi} ^2 y=0, & 0 <x <1, \, t  >0,\\
y(0,t) = y_x (0,t) =  0, & t > 0, \\
y_{xx} (1,t) = - \, f(y_{tx}(1,t)), & t > 0,\\
y_{xxx} (1,t) =\, f(y_{t}(1,t)), & t > 0,\\
y(x,0) = y_0(x), \, y_t (x,0) = y_1 (x), & 0<x<1,
\label{(1V)}
\end{array}
\right.
\end{equation}
or equivalently
\begin{equation}
\left \lbrace
\begin{array}{l}
\dot \phi(t) + A \phi(t) =0,\\
\phi(0) = \phi_0,
\end{array}
\right.
\label{(6.2)}
\end{equation}
where $\phi=(y,z)$, $z=y_t$ and $A$ is defined by (\ref{(5)})-(\ref{(6)}).

\medskip

We have the following proposition whose proof can be found in {\cite{CH:99} (see Lemma 5 and Proposition 3):}
\begin{prop}
Assume that the condition (\ref{ang}) holds, that is, $\mid  \varpi \mid < 3\sqrt{ EI/ \rho}$. Then, for any function $f$ satisfying solely the assumption {\bf H.I}, the operator $A$ defined by (\ref{(5)})-(\ref{(6)}) is m-accretive in  ${\mathcal H}$ with dense domain. Additionally, we have:

\medskip
\begin{enumerate}
\item
For any initial data $ \phi _0 = (y_0,y_1) \in {\mathcal D}(A)$, the system (\ref {(6.2)}) admits a unique solution $\phi (t)= (y(t),z(t)) \in {\mathcal D}(A)$ such that
$$ (y,y_t) \in L ^{\infty}(\R^+;{\mathcal D}(A)),\;\; \mbox{and}\;\;
\frac{d}{dt}(y,y_t) \in L ^{\infty}(\R^+;{\mathcal H}).$$
The solution $\phi=(y,z)$ is given by $\phi(t)=e^{- At}\phi _0,$ for all
$t\geq 0$
where $\left(e^{- At}\right)_{t \geq 0}$ is the semigroup generated by $- A$ on $\overline {{\mathcal D} (A)} = {\mathcal {H}} $.
Moreover, the function $ t \longmapsto \Vert  A \phi (t) \Vert_{\mathcal H} $
is decreasing.

\item
For any initial data $\phi _0 = (y_0,y_1) \in \overline {{\mathcal D} (A)} = {\mathcal {H}}$, the equation (\ref {(6.2)}) admits a unique mild solution $\phi(t)=
 e^{- At}\phi _0$ which is bounded on ${\R ^+}$ by $||\phi_0||_{\mathcal H}$ and
$$ \phi=(y,y_t) \in C^0 (\R^+; {\mathcal {H}} ).$$

\item The semigroup $\left(e^{- A t}\right)_{t \geq 0}$ is asymptotically stable in ${\mathcal H} $.

\label{(pp3)}
\end{enumerate}
\end{prop}

The first main result of this article is:

\begin{theo}
\label{Theop} Let us suppose that the conditions (\ref{ang}) and {\bf H.I}-{\bf H.II} hold, and let $(y_0,y_1) \in {\mathcal H}$. Then, there exist positive constants $k_{1}$, $k_{2}$, $k_{3}$ and $\varepsilon_{0}$ such that the energy $E(t)$ associated to the solution of (\ref{(1V)}) satisfies
\begin{eqnarray}
E(t) : = \frac{1}{2} \, \int_0^1 \left( \rho y_t^2(x,t) + EI y_{xx}^2(x,t) -\rho \varpi^2 y^2   \right) \, dx
\leq k_{3} \, H^{-1}_{1} \left( k_{1}t+
k_{2} \right), \,  \, \forall \, t\geq 0,   \label{theoremenergy1}
\end{eqnarray}
where
\[
H_{1}(t) = \int_{t}^{1} \frac{1}{H_{2}(s)}ds,\hspace{1cm} H_{2}(t)=t \, H^\prime (\varepsilon_{0} t).
\]
Here $H_{1}$ is a strictly decreasing and convex function on $(0,1]$, with $\displaystyle\lim_{t\rightarrow0}H_{1}(t)= +\infty$.
\end{theo}

Hereby, the remaining task is to provide a proof of the above theorem. To fulfill this objective, we shall establish several lemmas. 
Note that in view of a standard density argument together with the contraction of the semigroup, it suffices to prove Theorem \ref{Theop}  for a strong solution stemmed from a regular initial data in the domain  ${{\mathcal D} (A)}$.

First, we have
\begin{lemma}\label{lemmaenergie}
Let $y$ be a solution of the system (\ref{(1V)}). Then, the functional $E$ satisfies
\begin{equation}
E^{\prime}(t)=-EI y_{tx}(1,t)f(y_{tx}(1,t))-EI y_{t}(1,t)f(y_{t}(1,t)) \leq 0.
\label{energyde}
\end{equation}
\end{lemma}
\begin{proof}
Multiplying the first equation in (\ref{(1V)}), by $y_{t}$, using integrations by parts with respect to $x$ over $(0,1)$ along with the boundary conditions and the hypotheses {\bf H.I}, we obtain \eqref{energyde}.
\end{proof}

The second lemma is
\begin{lemma}\label{lemmaF}
Let $(y,y_t)$ be a solution of the system (\ref{(1V)}). Then, there exists $\theta >0$ such that the functional
\begin{equation}
F(t):= 2 \int_0^1 x y_t(x,t) y_x (x,t) \, dx,\;\; t\geq 0, \label{F}
\end{equation}
verifies the following estimate
\begin{eqnarray}\label{F'}
F^\prime(t) &\leq & y_t^2(1,t)+\big(1+2\theta\big)\frac{EI}{\rho}f^2(y_{tx}(1,t))+\theta\frac{EI}{\rho}f^2(y_{t}(1,t))
-K\| (y(t),y_t(t)\|_{\mathcal{H}}^2, \mbox{ a.e. } t\geq 0.
\end{eqnarray}
\end{lemma}

\begin{proof}
We differentiate $F$ with respect to $t$, we get after a straightforward computation:
\begin{eqnarray}\label{Fb}
F^\prime(t) &=& -2 \frac{EI}{\rho} y_{x}(1,t)f(y_{tx}(1,t))-2 \frac{EI}{\rho} y_{x}(1,t)f(y_{t}(1,t))+{ \varpi}^2y^2(1,t)+ \frac{EI}{\rho}f^2(y_{tx}(1,t))\\
&+&y_t^2(1,t)-\int_0^1 \left[y_t^2 (x,t) + 3\frac{EI}{\rho}y^2_{xx} (x,t) +{ \varpi}^2y^2 (x,t) \right] \, dx, \; \hbox{a.e.} \; t\geq0. \nonumber
\end{eqnarray}
Using Young's inequality, we obtain :
\begin{equation}\label{yf}
-2 \frac{EI}{\rho}y_{x}(1,t)f(y_{tx}(1,t)) \leq \frac{EI}{\rho\theta}y_x^2(1,t)+\theta\frac{EI}{\rho}f^2(y_{tx}(1,t)),
\end{equation}
and
\begin{equation}\label{yf1}
-2 \frac{EI}{\rho}y_{x}(1,t)f(y_{t}(1,t)) \leq \frac{EI}{\rho\theta}y_x^2(1,t)+\theta\frac{EI}{\rho}f^2(y_{t}(1,t)),
\end{equation}
for any $\theta>0$.\\
On the other hand, we have
$$ y_x(1,t) = y_x(1,t)-y_x(0,t)=\int_0^1 y_{xx}(x,t) \, dx \leq \|y_{xx}\|_{L^2(0,1)}.$$
Consequently, we obtain
\begin{equation}\label{yx}
y_x^2(1,t)\leq \int_0^1 y^2_{xx}(x,t) \, dx.
\end{equation}
Similarly, we use the fact $y(0,t)=0$ to get
\begin{equation}\label{y2}
y^2(1,t)\leq \frac{1}{3}\int_0^1 y^2_{xx}(x,t)\,dx.
\end{equation}
Combining (\ref{Fb})-(\ref{y2}) yields
\begin{eqnarray}\label{derF}
F^\prime(t) &\leq &y_t^2(1,t) +(1+2\theta)\frac{EI}{\rho}f^2(y_{tx}(1,t))+\theta\frac{EI}{\rho}f^2(y_{t}(1,t))\\
&+&\frac{1}{\rho}\int_0^1 \left[-\rho y_t^2 (x,t) +\left(\frac{2}{\theta}+\frac{{\rho \varpi}^2}{3EI}-3 \right)EIy^2_{xx} (x,t) -\rho{ \varpi}^2y^2 (x,t) \right] \, dx, \, \hbox{a.e.} \, t \geq 0.\nonumber
\end{eqnarray}
Recalling the condition (\ref{ang}), one can choose $\theta$ such that $\frac{2}{\theta}+\frac{\rho \varpi^2}{3EI}-3<0$. Then, we deduce the existence of a positive constant $K$ such that the inequality (\ref{F'}) is verified.
\end{proof}

Now, we are able to achieve the main objective of this section:
\begin{proof}[Proof of Theorem \ref{Theop}]
First of all, let
\begin{equation}\label{func}
E_0(t):= E(t)+ \varepsilon \, F(t)=\frac{1}{2}\| (y,y_{t})\|_{\mathcal{H}}^2+ \varepsilon \, F(t)
\end{equation}
in which $\varepsilon$ is a sufficiently small positive constant to be chosen later (see below (\ref{eo})) and $F(t)$ is defined by (\ref{F}). \\
It is easy to verify that the functional $E_0$ is equivalent to the energy $E$.\\
In fact, we have
$$|E_0(t)- E(t)|= \varepsilon \, |F(t)|\leq c\varepsilon \,E(t),$$
where $c$ is a positive constant.\\
On one hand, it follows from (\ref{F'}) and (\ref{func}) that
\begin{eqnarray}
E^\prime_0(t) &\leq & E'(t)+\varepsilon y_t^2(1,t)+\varepsilon\big(1+2\theta\big)\frac{EI}{\rho}f^2(y_{tx}(1,t))
+\varepsilon\theta\frac{EI}{\rho}f^2(y_{t}(1,t))\nonumber \\
&& -\varepsilon K E(t), \mbox{ a.e. } t\geq 0. \label{zz1}
\end{eqnarray}

Note also that we can suppose that $\max(r,f_0(r))<\sigma,$ where $r$ and $\sigma$ are defined in {\bf H.II} and (\ref{210}). Thereafter, we set $\upsilon=\min(r,f_0(r)),$ and we hence deduce from {\bf H.II} that,
$$\frac{f_0(\upsilon)}{\sigma}|s|\leq f_0(|s|)\leq f(|s|) \leq f_0^{-1}(|s|)\leq \frac{f_0^{-1}(\sigma)}{\upsilon}|s|,$$
for $\upsilon \leq |s| \leq \sigma.$ The latter implies that

\begin{equation}\label{H22}
\left\{
\begin{array}{l}
f_{0}(\vert s \vert)\leq \vert f(s) \vert \leq f_{0}^{-1}(\vert s \vert),
\hspace{1.65cm} \text{for all}\hspace{1.2em} \vert s\vert < \upsilon, \\
c'_{1} \vert s \vert\leq \vert f(s)\vert \leq c'_{2} \vert s\vert ,\hspace{
24mm}\text{ for all} \hspace{4mm}\vert s\vert\geq \upsilon,
\end{array}
\right.
\end{equation}
where $c'_{i} > 0,$ for i = 1, 2. Additionally, it follows from (\ref{210}) that $\displaystyle H(s^2)=|s|f_0(|s|)=sf_0(s)$ and $\displaystyle H(f^2(s))=|f(s)|f_0(|f(s)|),$ $\forall s\geq 0$.\\
In light of (\ref{H22}), we deduce that
$\displaystyle f_0(|f(s)|)\leq |s|,$ for $|s|\leq \upsilon$ and hence
$$H(f^2(s))\leq |s||f_0(s)|=sf_0(s), \; \mbox{ for } |s|\leq \upsilon.$$
Using the fact that $H$ is strictly convex on $[0,r^2]$, it follows that
$$H\left( \frac{1}{2}s^2+\frac{1}{2}f^2(s) \right) \leq H(s^2)+H(f^2(s))\leq sf_0(s),\; \mbox{ for } |s|\leq \upsilon.$$
Since $H^{-1}$ is strictly increasing in $[0,r^2]$, we conclude
\begin{equation}\label{H}
s^2+f^2(s)\leq 2H^{-1}(sf_0(s)),\; \mbox{ for } |s|\leq \upsilon.
\end{equation} 
On the other hand, by virtue of the assumption {\bf H.II}, Lemma \ref{lemmaenergie} and (\ref{H}), we obtain
$$
\begin{array}{lll}
&& \displaystyle{y_t^2(1,t)+\theta\frac{EI}{\rho}f^2(y_{t}(1,t)) +\big(1+2\theta\big)\frac{EI}{\rho}f^2(y_{tx}(1,t))}\\[3mm]
&\leq &\displaystyle e_1 \big( y_t^2(1,t)+ f^2(y_{t}(1,t)))1_{\{|y_{t}|\geq \upsilon\}} + e_1 \big( y_t^2(1,t)+ f^2(y_{t}(1,t)))1_{\{|y_{t}|\leq \upsilon\}}\\[2mm]
&& +\displaystyle{(1+2\theta)\frac{EI}{\rho} e_1 y_{tx}(1,t)f(y_{tx}(1,t))1_{\{|y_{tx}|\geq \upsilon\}}+(1+2\theta)\frac{EI}{\rho}\big( y_{tx}^2(1,t)+f^2(y_{tx}(1,t))\big)1_{\{|y_{tx}|\leq \upsilon\}}}\\[2mm]
&&-\displaystyle{(1+2\theta)\frac{EI}{\rho} y_{tx}^2(1,t)1_{\{|y_{tx}|\leq \upsilon\}}} \\[3mm]
& \leq & \displaystyle e_1 y_{t}(1,t)f(y_{t}(1,t))+ e_1 H^{-1}\left( y_{t}(1,t)f_0(y_{t}(1,t))\right) \, 1_{\{|y_{tx}|\leq \upsilon\}} \\[2mm]
& & +(1+2\theta)\frac{EI}{\rho} e_1 y_{tx}(1,t)f(y_{tx}(1,t))1_{\{|y_{tx}|\geq \upsilon\}}+\displaystyle{(1+2\theta)\frac{2EI}{\rho}H^{-1}\left( y_{tx}(1,t)f_0(y_{tx}(1,t))\right) \, 1_{\{|y_{tx}|\leq \upsilon\}}}\\[3mm]
& \leq & \! \! \displaystyle{- e_2 E'(t)+ e_1 H^{-1}\left( y_{t}(1,t)f_0(y_{t}(1,t))\right) \, 1_{\{|y_{tx}|\leq \upsilon\}}+ (1+2\theta)\frac{2EI}{\rho}H^{-1}\left( y_{tx}(1,t)f_0(y_{tx}(1,t))\right) \, 1_{\{|y_{tx}|\leq \upsilon\}}},
\end{array}
$$
in which $e_1 = \max(1,\theta \frac{EI}{\rho})$ and $e_2 = \max(e_1, (1+2\theta)\frac{2EI}{\rho})$ and the following partition of $(0,1)$ has been considered:
$$\{|y_{t}|\geq \upsilon\}=\{x\in(0,1) : \; |y_{t}|\leq \upsilon\}\;\mbox{ and }\;\{|y_{tx}|\leq \upsilon\}=\{x\in(0,1) : \; |y_{tx}|\leq \upsilon\}.$$
This, together with (\ref{zz1}), yields
\begin{eqnarray}\label{eo}
E_0^\prime(t) &\leq & \big(1-\varepsilon e_2 \big)E^\prime (t)- \varepsilon K\;E(t) +\varepsilon e_1\;H^{-1}\left( y_{t}(1,t)f_0(y_{t}(1,t))\right) \, 1_{\{|y_{tx}|\leq \upsilon\}}\nonumber\\
&& +\varepsilon(1+2\theta)\frac{2EI}{\rho}H^{-1}\left( y_{tx}(1,t)f_0(y_{tx}(1,t))\right) \, 1_{\{|y_{tx}|\leq \upsilon\}}\nonumber\\
&\leq & - \varepsilon K\;E(t) +\varepsilon e_1 H^{-1}\left( y_{t}(1,t)f_0(y_{t}(1,t))\right) \, 1_{\{|y_{tx}|\leq \upsilon\}}\nonumber\\
&& +\varepsilon(1+2\theta)\frac{2EI}{\rho}H^{-1}\left( y_{tx}(1,t)f_0(y_{tx}(1,t))\right) \, 1_{\{|y_{tx}|\leq \upsilon\}}, \mbox{ a.e. } t\geq 0,
\end{eqnarray}
provided that $\varepsilon$ is small so that $1-\varepsilon e_2 >0.$

Thereafter, for $\varepsilon_{0}<r^{2}$ to be chosen later, we have $H^{\prime }\geq0$ and $H^{\prime \prime}\geq 0 $ over $(0,r^{2}]$. Moreover, $E^{\prime }\leq0$. Consequently, the functional $R$ defined by
\[
R(t):=H^{\prime }\bigg(\varepsilon_{0}\frac{E(t)}{E(0)}\bigg)E_0(t)+\delta E(t),
\]
is equivalent to $E(t),$ for some $\delta >0$.
\\
In addition, we have $\varepsilon_{0}\frac{E^{\prime }(t)}{E(0)}H^{\prime \prime
}\left(\varepsilon_{0}\frac{E(t)}{E(0)}\right)E_0(t)\leq 0$. This, together with (\ref{eo}), allows us to conclude that
\begin{eqnarray}
R^{\prime }(t)\! \! \!&=&\! \! \!\varepsilon_{0}\frac{E^{\prime }(t)}{E(0)}%
H^{\prime \prime}\left(\varepsilon_{0}\frac{E(t)}{E(0)}\right) \, E_0(t) + H^{\prime
}\left(\varepsilon_{0}\frac{E(t)}{E(0)}\right)E_0^{\prime}(t) + \delta E^\prime(t)\nonumber\\
&\leq& -\varepsilon K E(t)H^{\prime }\left(\varepsilon_{0}\frac{E(t)}{E(0)}\right) + \varepsilon e_1 H^{\prime
}\left(\varepsilon_{0}\frac{E(t)}{E(0)}\right) H^{-1}\left( y_{t}(1,t)f_0(y_{t}(1,t))\right) \, 1_{\{|y_{tx}|\leq \upsilon\}}\nonumber\\
&& +\varepsilon(1+2\theta)\frac{2EI}{\rho}H^{\prime
}\left(\varepsilon_{0}\frac{E(t)}{E(0)}\right) H^{-1}\left( y_{tx}(1,t)f_0(y_{tx}(1,t))\right) \, 1_{\{|y_{tx}|\leq \upsilon\}} + \delta E^\prime(t).
\label{r1'}
\end{eqnarray}
Our objective now is to estimate the second and the third term in the right-hand side of \eqref{r1'}. For that purpose, we introduce, as in \cite{AB1}, the convex conjugate $H^{*}$ of $H$ defined by
\begin{equation}\label{conjugate}
H^{\ast}(s)=s(H^{\prime })^{-1}(s)-H((H^{\prime })^{-1})(s)\text{ for } s\in(0,H^{\prime }(r^{2})),
\end{equation}
and $ H^{*}$  satisfies the following Young inequality:
\begin{equation}\label{ableq}
AB\leq H^{\ast}(A)+H(B) \text{ for } A\in (0,H^{\prime}(r^{2})),\ B\in (0,r^{2}).
\end{equation}
Now, taking  first $A_1=H^{\prime }\left(\varepsilon_{0}\frac{E(t)}{E(0)}\right)$ and $
B_1=H^{-1}\left( y_{t}(1,t)f_0(y_{t}(1,t))\right)$ and second $A_2=H^{\prime }\left(\varepsilon_{0}\frac{E(t)}{E(0)}\right)$ and $
B_2=H^{-1}\left( y_{tx}(1,t)f_0(y_{tx}(1,t))\right)$, we obtain
\begin{eqnarray}
R^{\prime }(t)&\leq& -\varepsilon K E(t)H^{\prime }\left(\varepsilon_{0}
\frac{E(t)}{E(0)}\right) +\varepsilon e_1 \; H^{\ast}\left( H^{\prime}\left(\varepsilon_{0}
\frac{E(t)}{E(0)}\right)\right)+\varepsilon e_1 \; H\left( H^{-1}(y_{t}(1,t)f_0(y_{t}(1,t))\right)\nonumber \\&+&(1+2\theta) \frac{2EI}{\rho} H^{\ast}\left( H^{\prime}\left(\varepsilon_{0}
\frac{E(t)}{E(0)}\right)\right) \nonumber +(1+2\theta)\frac{2EI}{\rho}H\left( H^{-1}(y_{tx}(1,t)f_0(y_{tx}(1,t))\right)+\delta E'(t)\nonumber\\
&\leq &-\varepsilon K E(t)H^{\prime }\left(\varepsilon_{0}\frac{E(t)}{E(0)}\right)+
\varepsilon e_1 \;\varepsilon_{0}\frac{E(t)}{E(0)}H^{\prime }\left(\varepsilon_{0}\frac{E(t)}{E(0)}\right)- \varepsilon e_1 \; H\left(\varepsilon_{0}\frac{E(t)}{E(0)}\right)\nonumber\\
&&+\varepsilon e_1 \;y_{t}(1,t)f_0(y_{t}(1,t))+
\varepsilon(1+2\theta)\frac{2EI}{\rho} \;\varepsilon_{0}\frac{E(t)}{E(0)}H^{\prime }\left(\varepsilon_{0}\frac{E(t)}{E(0)}\right)- \varepsilon(1+2\theta)\frac{2EI}{\rho} \; H\left(\varepsilon_{0}\frac{E(t)}{E(0)}\right)\nonumber\\
&+&\varepsilon(1+2\theta)\frac{2EI}{\rho} \;y_{tx}(1,t)f_0(y_{tx}(1,t))+\delta E^\prime(t) \nonumber\\
&\leq &-\varepsilon K E(t)H^{\prime }\left(\varepsilon_{0}\frac{E(t)}{E(0)}\right) +
\varepsilon\varepsilon_{0}\Big((1+2\theta)\frac{2EI}{\rho}+ e_1 \Big)\frac{E(t)}{E(0)}H^{\prime }\left(\varepsilon_{0}\frac{E(t)}{E(0)}\right)\nonumber\\
&&- \varepsilon e_2 E'(t)+\delta E^\prime(t). \nonumber\label{r1'2}\hspace{2cm}
\end{eqnarray}
With a suitable choice of $ \varepsilon_{0} $ and $\delta$, we deduce from the last inequality that
\begin{equation}
R^{\prime }(t)\leq - \varepsilon\left( K E(0)-\varepsilon_{0}\Big((1+\theta)\frac{2EI}{\rho}+ e_1 \Big)\right) \frac{E(t)}{E(0)%
}H^{\prime }\left(\varepsilon_{0}\frac{E(t)}{E(0)}\right)\leq -k H_{2}\left(\frac{
E(t)}{E(0)}\right),\label{r1'3}
\end{equation}
where $k=\varepsilon \Big(K E(0)-\varepsilon_{0}\Big((1+\theta)\frac{2EI}{\rho}+ e_1 \Big)\Big)>0$ and $H_{2}(s)=sH^{\prime}(\varepsilon_{0}s)$.

Since $E(t)$ and $R(t)$ are equivalent, there exist positive constants $a_{1}$ and $a_{2}$ such that
$$a_{1}R(t)\leq E(t)\leq a_{2}R(t).$$
We set now $S(t)=\frac{a_{1}R(t)}{E(0)}$. It is clear that  $S(t)\sim E(t)$. Taking into consideration the fact that $H_{2}^{\prime}(t),\ H_{2}(t)>0$ over $(0,1]$ (this is a direct consequence of strict convexity of $H$ on $(0,r^{2}]$) we deduce from (\ref{r1'3}) that
\[
S^{\prime }(t)\leq-k_{1}H_{2}(S(t)),\hspace{0,5cm}\text{ for all }t\in \mathbb{R_{+}}, \]
with  $k_{1}  >0$.\\
Recall that $\displaystyle H_{1}(t) = \int_{t}^{1} \frac{1}{H_{2}(s)}ds$ and integrating the last inequality over $[0,t]$, we obtain
$$
H_{1}(S(t))\geq H_{1}(S(0))+k_{1}t.
$$
Finally, since $H_{1}^{-1}$ is decreasing (as for $H_{1}$), we have
\[
S(t)\leq H^{-1}_{1} \left( k_{1} t+ k_{2} \right), \hspace{1cm}\text{with} \ k_{2}>0.
\]
Invoking once again the equivalence of $E(t)$ and $R(t)$, we deduce the desired result (\ref{theoremenergy1}).
\end{proof}

\begin{remark} As pointed out in the Introduction, the presence of the nonlinear moment control which involves $y_{tx}(1,t)$ is not necessary for the result stated in Theorem 3.1. In fact, the reader can easily check that the estimates obtained in the previous section remain valid as long as the nonlinear force control involves the velocity term $y_{t}(1,t)$.
\label{nrem}
\end{remark}

\section{Examples} \label{exa}
The objective of this section is to apply the inequality in \eqref{theoremenergy1} on some examples in order to show explicit stability results in term of asymptotic profiles in time. To proceed, we choose the function $H$ strictly convex near zero.

\subsection{Example 1}
Let $f$ be a function that satisfies
\[
c_{3}\vert s\vert^{p}\leq \vert f(s)\vert \leq c_{4}\vert s\vert^{\frac{1}{p}},
\]
with some $c_{3}$, $c_{4}>0$ and $p>1$. \\
For $f_{0}(s)=cs^{p} $, hypothesis {\bf H.II} is verified.
Then $H(s)=cs^{\frac{p+1}{2}}$, $H_{2}(s)=
\hat{c}_1 s^{\frac{p+1}{2}} $
 and
\[
H_{1}(t)=\int_{t}^{1}\frac{1}{c_1}s^{-\frac{p+1}{2}}\ ds =\frac{1}{
\hat{c}_2}(1-t^{-\frac{p-1}{2}}),\]
where $\hat{c}_i >0$, for $i=1,2$. Therefore,
\[
H_{1}^{-1}(t)=(\hat{c}_2 t+1)^{-\frac{2}{p-1}}.
\]
Using again (\ref{theoremenergy1}), we obtain
\[
E_0(t)\leq H_{1}^{-1}( k_{1}t +k_{2})=(
\hat{c}_2 ( k_{1}t +k_{2})+1)^{-\frac{2}{p-1}}.
\]

\begin{remark}\label{rmk}
Note that the exponential decay rate result has been obtained in \cite{CH:99}, for the case $f(s)=cs.$
\end{remark} 
\subsection{Example 2}
Let $f_{0}(s)=\frac{1}{s}\exp(-\frac{1}{s^{2}})$.
{Arguing} as in example 1, one can conclude that the energy of (\ref{(1V)}) satisfies
$$E_0(t)\leq \varepsilon \left( \ln \left(\frac{k_{1}t +k_{2}+c\exp(\frac{1}{\varepsilon_{0}})}{c}\right)\right)^{-1}.$$

\section{Stability of the closed-loop system}\label{sta}
Throughout this section, we shall assume without loss of generality that the physical parameters $EI, \, \rho$ are unit. Additionally, we shall suppose that (\ref{ang}), {\bf H.I} and {\bf H.III} are fulfilled. Moreover, the subsystem (\ref{(1V)}) is exponentially stable for the case $f_{0}(s)=cs$ (see Remark \ref{rmk}). This implies that there exist positive constants $M$ and $\eta$ such that
\begin{equation}
\| e^{-A t} \|_{\mathcal L(\mathcal{H})} \leq M e^{-\eta t}
, \;\; \forall t \geq 0.
\label{e}
\end{equation}

The aim of this section is to show the exponential stability of the closed-loop system (\ref{(6.0)}). We shall use the same arguments as put forth in \cite{CH:99} but with a number of changes born out of necessity.

First of all,  thanks to the assumptions (\ref{ang}), {\bf H.I} and {\bf H.III}, it follows from Proposition 3.1 that for any initial data $ \left( \phi , \omega \right) \in {\mathcal D} ({\mathcal A})$, the system has a unique global bounded solution on  $(0,\infty)$ \cite{CH:99}. For sake of completeness, we shall provide a sketch of the proof of this result. First, we know from \cite[Proposition 2 p. 526]{CH:99} that, thanks to the assumption {\bf H.I}, the operator $ A$ defined by (\ref{(5)})-(\ref{(6)}) is m-accretive in ${\mathcal H}$ with dense domain. This implies that the nonlinear operator $\mathcal{A}$ (see (\ref{(6.1)})) is also m-accretive with dense domain in $\mathcal{X}$. Additionally,  the operator $\mathcal{B}$ defined in (\ref{(6.1)}) is Lipschitz on bounded subsets of $\mathcal{X}$. Thereby, the closed-loop system (\ref{(6.0)}) has a local solution $(\phi, \omega )(t) =(y(t),z(t),\omega(t))$, for $t \in [0,T]$. Next, it suffices to prove that the solution is indeed global. To do so,
consider  the functional
\begin{eqnarray}
{\mathcal V} (t) &=& \displaystyle \frac{1}{2}(\omega(t)  - \varpi )^2 \int_0^1 y^2 dx +  \frac{I_d}{2} \; (\omega(t)  - \varpi )^2  \label{ya1} + \displaystyle \frac{1}{2} \int_0^1 \left \{  y_t^2 + y_{xx} ^2 - \varpi ^2  y^2 \right \} dx. \label{(yassir)}
\end{eqnarray}
Then, one can easily check that $\mathcal{V}$ is a Lyapunov function and satisfies along the regular solutions of (\ref{(6.0)})
\begin{equation}
\barr
\displaystyle \frac{d}{dt} {\mathcal V} (\phi,\omega)(t) =
-\displaystyle EI y_{tx}(1,t)f(y_{tx}(1,t))-EI y_{t}(1,t)f(y_{t}(1,t)) -(\omega (t)- \varpi  )\; \gamma \Bigl( \omega (t)-\varpi  \Bigr) \leq 0, \, \forall t \geq 0,
\label{(yassir*)}
\ear
\end{equation}
by means of the conditions {\bf H.I} and {\bf H.III}. Therefore, the solution of (\ref{(6.0)}) must be global.

Secondly,  we consider the subsystem
$$\left \{
\barr
\dot \psi(t) + A \psi(t)+ B (t,\psi(t))=0, \; t>0,\\
\psi(0)=\psi_0 , \ear \right. $$
where $\psi=(y,z)$, $\psi_0 \in {\mathcal D}(A)$ and
\beq
B (t,\psi )=( {\varpi}^2 -{\omega}^2 (t) ) P(\psi ),
\label{ka}
\eeq
in which $P$ is the compact operator on ${\mathcal H}$ defined by $P (y, z)=(0,y), \; \forall (y, z) \in {\mathcal H}.$ Moreover, $\omega (t)$ is the second component of the solution $(\psi(t),\omega(t))$ of the global system (\ref{(6.0)}) with initial data $(\psi_0,\omega_0) \in {\mathcal D}(\mathcal A) $. Accordingly, the solution $\psi(t)$ can be written as follows
\begin{equation}
\psi (t)= e^{-(t-T_0)A} \psi (T_0) + \int_{T_0}^t e^{-(t-\nu){A}} ( \omega^2(\nu) - \varpi^2  ) P \psi (\nu)\, d \nu,
\label{eq}
\end{equation}
for any $t \geq T_0$.

Clearly, it follows from (\ref{ka}) that
$$ \| B (t,\psi(t)) \|_{{\mathcal H}} \leq
|{\varpi}^2 -{\omega }^2 (t)| \|\psi(t) \|_{{\mathcal H}}.$$
Furthermore, in the light of (\ref{(yassir)}), (\ref{(yassir*)}) and the assumptions {\bf H.I}-{\bf H.III}, we have
$$\varpi -{\omega (\cdot)} \in L ^2 \left( [0, \infty[;\R \right) \cap L ^{\infty}
\left( [0, \infty[;\R \right) $$
and hence $\lim_{t \to +\infty} \omega(t) =\varpi $ \cite{CH:99}. Herewith, for any $ \epsilon >0$, there exists $T_0$ sufficiently large such that for each $t \geq T_0$, we have
\begin{equation}
| \omega^2(t) - \varpi^2 | <  \epsilon.
\label{li}
\end{equation}
Combining (\ref{li}) as well as (\ref{e}) and (\ref{eq}) and applying Gronwall's Lemma yields
\begin{equation}
\| \psi (t) \|_{\scriptscriptstyle {\mathcal H}} \leq M  e^{\scriptscriptstyle -(\eta-\epsilon M) (t-T_0)} \| \psi (T_0) \|_{\scriptscriptstyle {\mathcal H}},
\label{29..}
\end{equation}
for any $t \geq T_0$. Thereafter, one can choose $\epsilon$ so that $\eta-\epsilon M >0$, and thus the solution $\psi (t)$ of (\ref{(6.0)}) is exponentially stable in ${\mathcal H}$. Finally, amalgamating these properties and going back to (\ref{(6.0)}) one can show the exponential convergence of $\omega(t) -{\varpi}$ in $\R$ by means of {\bf H.III}. The proof runs on much the same lines as that of \cite[Theorem 2]{CH:99}. Indeed, we deduce from (\ref{(6.0)}) that
\begin{equation}
\frac{d}{dt}  \left( \omega (t) -\varpi \right) = \frac{\gamma \left( {\omega (t)}-{\varpi} \right)  \| y\|_{\scriptscriptstyle L^2(0,1)}^2 -2{I_d} \, \omega (t) <y,y_{t}>_{\scriptscriptstyle L^2 (0,1) }}{I_d \left( I_d + \| y\|_{\scriptscriptstyle L^2 (0,1)}^2\right)} - \frac{1}{{I_d}}  \gamma \left({\omega (t)}-{\varpi} \right) . \label{wo}
\end{equation}
Multiplying (\ref{wo}) by $\omega(t) -{\varpi}$ and using the assumption {\bf H.III}, we get
\begin{eqnarray}
\dfrac{1}{2} \frac{d}{dt}  \left( \omega (t) -\varpi \right)^2 &\leq & \frac{ \left( {\omega (t)}-{\varpi} \right) \gamma \left( {\omega (t)}-{\varpi} \right)  \| y\|_{\scriptscriptstyle L^2(0,1)}^2 -2{I_d} \, \omega (t) \left( {\omega (t)}-{\varpi} \right) <y,y_{t}>_{\scriptscriptstyle L^2 (0,1) }}{I_d \left( I_d + \| y\|_{\scriptscriptstyle L^2 (0,1)}^2\right)} \nonumber\\
&& -\frac{\kappa}{{I_d}}   \left( {\omega (t)}-{\varpi} \right)^2. \label{wo1}
\end{eqnarray}

Next, solving the above inequality, we get
\begin{eqnarray}
&& \left( \omega (t) -\varpi \right)^2 \leq   e^{\scriptscriptstyle \frac{-2 {\color{red}\kappa} (t-T_0)}{I_d}} ({\omega (t)}-{\varpi})^2 \label{029} \\
&& + \dfrac{1}{I_d}\displaystyle \int_{T_0}^t {e^{\scriptscriptstyle \frac{-2 {\color{red}\kappa} (t-s)}{I_d}}} \left( {\omega (s)}-{\varpi} \right) \left\{ \frac{2 \gamma \left( {\omega (s)}-{\varpi} \right)  \| y (s) \|_{\scriptscriptstyle L^2 (0,1)}^2 }{I_d + \| y (s)\|_{\scriptscriptstyle L^2 (0,1)}^2}  -\dfrac{4{I_d}  \omega (s)  <y (s),y_{s}(s)>_{\scriptscriptstyle L^2 (0,1) }}{I_d + \| y (s)\|_{\scriptscriptstyle L^2 (0,1)}^2} \right\} d s .  \nonumber
\end{eqnarray}
By virtue of the Lipschitz property of $\gamma$ (see {\bf H.III})  and using (\ref{29..}) as well as the boundedness of solutions, one can write (\ref{029}) after a simple calculation as follows
$$|{\omega (t)}-{\varpi}| \leq a e^{-b (t-T_0)},$$
for any $ t \geq T_0$ and for some positive constants $a$ and $b$.

To recapitulate,  we have proved the second main result:
\begin{theo}\label{exp}
Assume that (\ref{ang}), {\bf H.I}  and {\bf H.III} are fulfilled.  We also suppose that {\bf H.II} holds with $f_{0}(s)=cs$. Then, for each initial data $ \left( \phi _0, \omega _0 \right) \in {\mathcal D} ({\mathcal A}) $ the solution $ \left( \phi (t), \omega (t)  \right) $ of the closed-loop system (\ref{(6.0)}) 
exponentially tends to $(0, \varpi)$ in $ {\mathcal X} $ as $ t \to + \infty$.
\label{so2}
\end{theo}
\begin{remark}
The authors believe that if the semigroup $e^{- A t}$ has a polynomial decay, then the original system (2.3) will be also polynomially stable. In fact, if $f_0(s)=cs^p$, with $p>1$, then the energy $E_0(t)$ has a rational decay rate (see Example 1 of Subsection 4.1) and in such a case the system (2.3) is expected to be polynomially stable. This case is worth investigating in a future work as the problem remains open.
\end{remark}

\section{Conclusion}
This note was concerned with the stabilization of the rotating disk-beam system. A feedback law, which consists of a nonlinear torque control exerted on the disk and  a nonlinear boundary control applied on the beam, is put forward. Then, it is shown that the closed-loop system is exponentially stable provided that the angular velocity of the disk does not exceed a well-defined value and the nonlinear functions involved in the feedback law obey reasonable conditions. This result improves the previous one in \cite{CH:99}, where the functions involved in the boundary control must satisfy strong and restrictive conditions such as  almost linear property.

\section*{Acknowledgment}
The valuable suggestions and comments from the editor and the anonymous referees are greatly appreciated.

\end{document}